\documentclass[11pt]{article}
\pdfoutput=1
\usepackage[margin=1in]{geometry}
\usepackage{latexsym, color}
\usepackage{float}
\newcommand\qed{\hfill$\Box$}
\newcommand\proof{{\it Proof.}\nobreak\\}
\newtheorem{tm}{Theorem}
\newtheorem{lm}{Lemma}

\newtheorem{rem}{Remark}
\newtheorem{con}{Conjecture}
\def\F{{\cal F}}

\title{New infinite family of regular edge-isoperimetric graphs}
\author{Sergei L. Bezrukov \and Pavle Bulatovic \and Nikola Kuzmanovski} 
\date{\vspace{-5ex}}

\begin{document}
\maketitle

\begin{abstract}
We introduce a new infinite family of regular graphs admitting nested 
solutions in the edge-isoperimetric problem for all their Cartesian powers. 
The obtained results include as special cases most of previously known 
results in this area.
\end{abstract} 

\section{Introduction}

Let $G=(V_G,E_G)$ be a graph and $A,B\subseteq V_G$. Denote
\begin{eqnarray*}
   I_G(A,B)&=&\{(u,v)\in E_G\bigm| u\in A,\; v\in B\}, \\ 
   I_G(A)&=& I_G(A,A), \\
   I_G(m)&=&\max_{A\subseteq V_G,\; |A|=m} |I_G(A)|. 
\end{eqnarray*}
We will often omit the index $G$.
Our subject is the following version of the {\it edge-isoperimetric problem\/}
(EIP): for a fixed $m$, $1\leq m\leq |V_G|$, find a set $A\subseteq V_G$ 
such that $|A|=m$ and $|I(A)|=I(m)$. We call such a set $A$ {\it optimal\/}. 
This problem is known to be NP-complete in general and has many applications 
in various fields of knowledge, see survey \cite{B2}. 

We restrict ourselves to graphs representable as Cartesian products of other 
graphs. Given two graphs $G=(V_G,E_G)$ and $H=(V_H,E_H)$, their
{\it Cartesian product\/} is defined as a graph $G\times H$
with the vertex-set $V_G\times V_H$ whose two vertices
$(x,y)$ and $(u,v)$ are adjacent iff either $x=u$ and $(y,v)\in E_H$, or
$(x,u)\in E_G$ and $y=v$. The graph $G^n=G\times G\times\cdots \times G$
($n$ times) is called the $n^{th}$ {\em Cartesian power} of $G$.

A particular interest in study of EIP is the case
when there exists a total order $\cal O$ on the vertex set of graphs in 
question such that for every $m$ the initial segment of 
this order of size $m$ is an optimal set. Such order $\cal O$ is called {\em
optimal order}. There exist graphs such that their Cartesian powers do not
admit optimal orders. 
For example, it is known that there does not exist optimal orders for the 
second and higher powers of cycles of length $p$ for $p> 5$ \cite{C}.
However, existence of a nested structure of solutions (that is, an optimal
order) is an important graph
property, because it provides as an immediate consequence solutions to many
applied problems. Among such problems are the cutwidth, wirelength, and
bisection width problems, construction of good $k$-partitioning of graphs
and their embedding to some other graphs \cite{B2}. This stimulates the study
of graphs which admit optimal orders for all their Cartesian powers. We
call such graphs {\em edge-isoperimetric}.

The EIP for the Cartesian powers of a graph $G$ has been well intensively
studied for various graphs, see survey \cite{B2}. To summarize some of these 
results and present our new one we need to define the 
{\it lexicographic order\/} on a set of $n$-tuples with integer entries. 
For that we say that $(x_1,\dots,x_n)$ is greater than $(y_1,\dots,y_n)$ 
iff there exists an index $i$, $1\leq i\leq n$, such that $x_j=y_j$ for 
$1\leq j<i$ and $x_i>y_i$. 
It turns out that just a few different optimal orders are discovered for the
EIPis. Some of them are proved to work just for a few 
graphs \cite{B2, BPP}. This leaves the lexicographic order to work in most of
known cases. In large this is due to the following result called in \cite{AC}
{\em local-global principle}.

\begin{tm} \label{lgp}
{\rm (Ahlswede, Cai \cite{AC})}
If lexicographic order is optimal for $G\times G$ then it is optimal for
$G^n$ for any $n\geq 3$.
\end{tm}

The main difficulty in applying the local-global principle to a given graph
$G$ is to establish the optimality of the lexicographic order for $G\times G$.
For this, however, no general methods have been developed so far.
At this state of research a solution of the EIP on $G\times G$ for any other
concrete graph $G$ would be interesting and useful for developing more
general methods.
It seems difficult to characterize all graphs for whose all Cartesian powers 
the lexicographic order is optimal. Examples include graphs studied in
\cite{AC, BE, H2, L}. However, we believe to the following conjecture.

\begin{con}
If lexicographic order is optimal for $G\times G$ then $G$ is regular.
\end{con}

All graphs studied in the above mentioned papers \cite{AC, BE, H2, L} are
regular. In the light of this conjecture we emphasize on Cartesian powers 
of regular graphs. It is more convenient to work not directly with graphs, 
but with their numeric characteristic $\delta_G$ that we call 
{\em $\delta$-sequence}. For a graph $G=(V,E)$ denote
\begin{eqnarray*}
   \delta(m) &=& I(m) - I(m-1), \mbox{ with } \delta(1) = 0,\\
   \delta_G &=& (\delta(1),\delta(2).\dots,\delta(|V|)).
\end{eqnarray*}
For $|V_G|=p$ we call $\delta_G$ {\em symmetric} if 
  $$ \delta(i) + \delta(p-i+1) = \delta(p)\qquad\mbox{for } i=1,\dots,p.$$
For example, the $\delta$-sequence of the 3-dimensional unit
cube $\delta_{Q^3} = (0,1,1,2,1,2,2,3)$ is symmetric.
It is easily shown that if $G$ is regular then $\delta_G$ is symmetric. The
recent result shows that the converse is also true.
\begin{tm} \label{reg}
{\rm (Bonnet, Sikora \cite{BS})}
If $\delta_G$ is symmetric then $G$ is regular.
\end{tm}

Our experience shows that in order for the lexicographic order to be optimal 
for $G\times G$, the graph $G$ has to be dense, that is, have many edges. It
seems that high density and regularity are crucial conditions for the
lexicographic order to be optimal. Cliques have highest density and the
lexicographic order is optimal for every their Cartesian power \cite{L}. It is
interesting to mention that removal of an edge from a clique leads to the
next best choice for a high density graph, but products of this graph do not
admit any optimal order. A natural way to construct dense regular graphs is to 
start with a clique $K_p$ (or $K_{p,p}$) and remove a factor from it. 
In particular, one can consider removing $s\geq 1$ disjoint perfect matchings 
$M$ from $K_p$ or $K_{p,p}$. These are exactly the graphs being studied in 
\cite{AC, BE, H2, L}. The $\delta$-sequences of these graphs are as follows:
\begin{itemize}
\item $\delta_{K_p} = (\{0,1,2,3,\dots,p-1\})$
\item $\delta_{K_p - sM} = (\{0,1,2,\dots,\frac{p}{2}-1\},\;\; \{\frac{p}{2}-s,
\frac{p}{2}-s+1, \frac{p}{2}-s+2,\dots,p-s-1\})$
\item $\delta_{K_{p,p}} = (\{0,1\},\; \{1,2\},\; \{2,3\},\dots,\{p-1,p\})$
\end{itemize}
The braces above show partitioning of $\delta$-sequences into maximum
monotonic subsequences. Thus, $K_p$ has just one monotonic subsequence,
$K_p - sM$ (a clique with $s$ disjoint perfect matchings removed) has 2 
(maximum) monotonic subsequences, and $K_{p,p}$ has $p$ ones. Our main 
result generalizes all the above mentioned results for graphs $H_{s,p,i}$
defined by their $\delta$-sequences. Namely,
$\delta_{H_{s,p,i}}$ must admit partitioning into $s$ monotonic subsequences 
of size $p$ of the form:
$$(\{0,1,\dots,p-1\},\;\;\{p-i,\dots, p-i+(p-1)\},\;\dots,\;\{(s-1)(p-i),
\dots, (s-1)(p-i) + p-1\}).$$
For example, for $s=3$, $p=4$, $i=2$ one has
$\delta_{H_{3,4,2}} = (\{0,1,2,3\},\;\{2,3,4,5\},\;\{4,5,6,7\})$.

In the next section we present a construction of some graphs $H_{s,p,i}$
with such $\delta$-sequences. Our main result is as follows: 

\begin{tm}
{\rm (Main result)}\label{main}
Lexicographic order is optimal for $H_{s,p,i}\times H_{s,p,i}$ for 
$p\geq 3$, $s\geq 2$, and $1\leq i\leq p-i$.
\end{tm}

Due to Theorem \ref{lgp} this result is also valid for $H^n_{s,p,i}$ for
$n\geq 2$. Note that the family of graphs $H^n_{s,p,i}$ include cliques,
cliques without $s$ perfect matchings, and also $t$-partite graphs
$K_{p,\dots,p}$. Thus, the only published family of regular graphs that
admit optimal orders for all their Cartesian powers and that are not covered 
by our main result are complete bipartite graphs with $s$ perfect matchings 
removed \cite{BE}.

The paper is organized as follows.
In the next section we present a general construction of regular graphs
admitting optimal orders and show that the graphs $H_{s,p,i}$ for some values
of parameters can be constructed this way. In Section 3 we explore some 
properties of graphs $H_{s,p,i}$ and present some auxiliary results used in
the proof of our main result in Section 4. Some computational experiments 
and concluding remarks are outlined in Section 5.

\section{Clique structure of optimal orders}

Throughout this section we assume that $G=(V,E)$ is a graph admitting an
optimal order.
We start with some basic properties of the $\delta$-sequence for a graph
$G=(V,E)$. Obviously, $\delta_G(i)\geq 0$ for all $i\in \{1,\dots,|V|\}$.
Moreover the strict inequality for $i\geq 2$ holds iff $G$ is connected.

\begin{lm}\label{lem1}
Let $G=(V,E)$ be a graph that admits nested solutions.
Then $\delta_G(i+1) -\delta_G(i)\leq 1$ for all $i\in \{1,\dots |V|\}$.
\end{lm}
\begin{proof}
Assume to the contrary that there exists $j$, $1\leq j\leq |V|$, such that
$\delta(j+1)-\delta(j)>1$, that is, $\delta_G(j+1)\geq \delta_G(j)+2$.
Let $x_1$ and $x_2$ be the $j$-th and $j+1$-th vertices in the optimal order
$\cal O$ of $G$ and $X$ be the set of vertices preceeding $x_1$ in this order.
Then $|I(X,\{x_1\})|=\delta(j)$. Since $\delta_G(j+1)\geq \delta_G(j)+2$ 
we have $|I_G(X,\{x_2\})|=\delta_G(j+1)\geq \delta_G(j)+1$.
Hence, $|I_G(X\cup\{x_2\})| > |I_G(X\cup\{x_1\})|$, which contradicts the
optimality of order $\cal O$.
\hfill\qed
\end{proof}

We call a subsequence of consecutive entries $\delta(a),\dots,\delta(b)$ of
$\delta_G$ {\em monotonic segment} if $\delta(a)<\cdots<\delta(b)$ and it is
longest with respect to this property. This way $\delta_G$ can be partitioned
into monotonic segments. Denote by $M_{G,i}$ the set of vertices
corresponding to the entries of the $i$-th monotonic segment. We call
$M_{G,i}$ the {\em $i$-th monotonic set}.

\begin{tm}\label{thm1}
Let $G=(V,E)$ admit nested solutions and $M_{G,i}$ be the $i^{th}$ monotonic 
set corresponding to the monotonic segment $\delta(a),\dots,\delta(b)$.
Denote $X=\cup_{s=1}^{i-1}M_{G,s}$.
Then $M_{G,i}$ induces a clique in $G$ and $\forall u\in M_{G,i}$ it holds
$|I_G(X,\{u\})|=\delta(a)$.
\end{tm}
\begin{proof}
Let $V=\{v_1,v_2,\dots , v_{|V|}\}$, where the vertices are labeled according
to an optimal order for $G$.
Denote $S_k=\{v_j\in M_{G,i}\bigm| j\leq a+k \}$
We prove the theorem by induction on $k$.

For $k=0$ the set $S_1=\{v_a\}$ obviously induces a clique. Moreover,
$|I(X,S_1)|=\delta(a)$. Assume that $k\geq 1$ and that the theorem is true 
for all $k'<k$. 

Since $\delta(a),\dots\delta(b)$ is a monotonic segment, one has
$\delta(a+k)=\delta(a)+k$. If $S_k$ is not a clique then 
$|I_G(S_k, \{v_{a+k}\})| < k-1$, which implies 
$|I_G(X,\{v_{a+k}\})|\geq \delta(a)+1  > |I_G(X,\{v_a\})|$. This, in turn,
implies $|I_G(X\cup \{v_{a+k}\})| > |I_G(X\cup \{v_a\}|$ which contradicts
the optimality of the vertex order.  \hfill\qed
\end{proof}

The following observation will be used in the proof of Theorem \ref{thm2}.

\begin{rem}\label{rem1}
If $M_{G,i}$ and $M_{G,j}$ are monotonic sets with $i<j$ and
$X\subseteq \cup_{s=1}^{i-1}M_{G,s}$, then for any $u\in M_{G,i}$ and $v\in
M_{G,j}$ it holds $|I_G(X,\{u\})|\geq |I_G(X,\{v\})|$. 
\end{rem}

For graphs $G_1=(V_1,E_1)$ and $G_2=(V_2,E_2)$ admitting optimal orders
denote by $G_1\circ G_2$ the {\em composition} of $G_1$ and $G_2$. That is,
the graph $H=(V_H,E_H)$ with 
\begin{eqnarray*}
V_H &=& V_1\cup V_2,\\
E_H &=& E_1\cup E_2\cup \{(a,b)\bigm | a\in V_1 \mbox{ and } b\in V_2\}.
\end{eqnarray*}

Assuming $G$ admits an optimal order $\cal O$,
denote by ${\cal F}_G(i)$ the set of the first $i$ vertices of $V_G$ in the
order $\cal O$. Further denote by $P_{G,i}=(V_i, E_i)$ the subgraph 
of $G$ induced by the vertex set ${\cal F}_G(i)$. We call $P_{G,i}$ the 
{\em $i$-th partial} of $G$. It is easily seen that the order $\cal O$ is 
optimal for $P_{G,i}$ for $i=1,\dots, |V_G|$.
Denote by ${\cal P}_G=\{P_{G,i}\bigm | 0\leq i\leq |V|\}$ the set of
all partials of $G$. The next theorem presents a general construction of
graphs admitting an optimal order. This construction will be then used to
construct graphs $H_{s,p,i}$.

\begin{tm}\label{thm2}
Let a graph $G$ admit nested solutions and 
$H_1,H_2,\dots, H_n\in {\cal P}_G$ with $H_i=(V_i,E_i)$ and
$|V_1|\geq |V_2|\geq \cdots \geq |V_n| > 0$. Further let $\{M_{{H_i},s}\}$,
be the set of monotonic sets of $H_i$ and the vertices of
each $H_i$ be ordered in its optimal order ${\cal O}_i$. Then the following
order $\cal O$ is optimal for $S_n=H_1\circ \cdots \circ H_n$: 
for $u\in M_{{H_i},s}$ and $v\in M_{{H_j},t}$ we write $u <_{\cal O} v$ iff
\begin{itemize}
\item[$(i)$] $s < t$, or
\item[$(ii)$] $s = t$ and $i < j$, or
\item[$(iii)$] $s = t$, $i = j$ and $u <_{{\cal O}_i} v$
\end{itemize}
\end{tm}
 
\begin{proof}
We prove the theorem by induction on $n$.
For $n=1$ it is obviously true. Assume $n\geq 2$ and that the theorem 
is true for all $n'<n$.

Let $H_1\circ \cdots \circ H_n=(V,E)$ and $A\subset V$.
Denote $U=V_n\cap A$ and $D=V_{S_{n-1}}\cap A$.
We transform $A$ by replacing $U$ with $U'={\cal F}_{H_n}(|U|)$ and
$D$ with $D'={\cal F}_{S_{n-1}}(|D|)$. Denote the resulting set by $B$.
Since $|I_{S_{n-1}}(D')|\geq |I_{S_{n-1}}(D)|$ and 
$|I_{H_n}(U')|\geq |I_{H_n}(U)|$ by induction, taking into account that
every vertex of $S_{n-1}$ is connected to all vertices of $V_n$, we conclude
$|I_{S_n}(B)| - |I_{S_n}(A)| \geq 0$.

Let $M_{H_n,q}$ be the first monotonic set such that 
$|U'\cap M_{H_n,q}| < |M_{H_n,q}|$ and 
let $M_{H_p,w}$ be the last monotonic set such that 
$D'\cap M_{H_p,w}\neq \emptyset$.

{\it Case 1:} Assume $q=w$. Theorem \ref{thm1} implies that the set
$T=\cup_{1\leq i\leq n} M_{H_i,q}$ is a clique in $S_n$. Since
$I_{S_n}(\cup_{j<q} M_{H_i,j},\{x\})$ is the same for every $i$ and 
$x\in M_{H_i,q}$, the number of inner edges of $S_n$ will not change if we
replace in $T\cap B$ with any subset of $T$ of the same size. 
Hence, we can transform $B$ into initial segment of order $\cal O$ without 
decreasing the number of its inner edges.

{\it Case 2:} Assume $q<w$, as argument in the case $q>w$ is similar.
Let $a\in M_{H_n,q}\setminus B$ and $b\in M_{H_p,w}\cap B$ for some
$p$. We transform $B$ to the set $C=(B\setminus \{b\})\cup \{a\}$.
Denote $B_1=\cup_{i=1}^{q-1}M_{H_n,i}$ and $B_2=\cup_{i=1}^{q-1}M_{H_p,i}$.
By Remark \ref{rem1} we have $|I_{S_n}(B_1,\{a\})|\geq |I_{S_n}(B_2,\{b\})|$.
Since $|B_1|=|B_2|$ we have $|I_{S_n}(B_1,\{b\})|=|I_{S_n}(B_2,\{a\})|$.
Denote $B_3=\cup_{s=q}^{w}M_{H_p,s}$.
Since $a$ is connected to all vertices of $S_{n-1}$
we have $|I_{S_n}(B_3\cap C,\{b\})|\leq |I_{S_n}(B_3\cap C,\{a\})|$.
Since $M_{H_n,q}$ is a clique and $b$ is connected to all vertices of $H_n$
we have $|I(M_{H_n,q}\cap B, \{a\})|=|I(M_{H_n,q}\cap B, \{b\})|$.
Since $a$ and $b$ are connected to all vertices of $S_n\setminus (H_n\cup
H_p)$ we have $|I_{S_n}(B\setminus (H_n\cup H_p), \{a\})|=
|I_{S_n}(B\setminus (H_n\cup H_p), \{b\})|$.
Hence, $|I_{S_n}(C)|\geq |I_{S_n}(B)|$. This way we can keep moving 
vertices one by one until we get an initial segment of order $\cal O$.
\hfill\qed
\end{proof}

As an immediate application of this result we provide a construction for the
graphs $H_{s,p,i}$ for some sets of parameters.

\begin{tm}\label{thm3}
Let $p\in N$ and $F$ be the set of all factors of $p$.
Then $H_{s,p,i}$ exists for every $i\in F$.
\end{tm}
\begin{proof}
Let $i\in F$ and
$G=(V, E)$ be a disjoint union of $s$ cliques, each of size $i$.
It is easily shown that $\delta(G) = (\{0,1,\dots,i-1\},\;\{0,1,\dots,i-1\},\;
\dots\;\{0,1,\dots,i-1\})$, where the sequence in braces repeats $i$ times.

Theorem \ref{thm2} implies that the composition  
$H = G\circ\dots\circ G$ ($p/i$ times) admits nested solutions. 
Taking into account the optimal order for $H$ stated in the theorem, 
we conclude $\delta_H=\delta_{H_{s,p,i}}$.
\hfill\qed
\end{proof}

\section{Some auxiliary results}

Let $G = (V, E)$ be a graph admitting an optimal order. We label its 
vertices with $0,1,\dots, |V|-1$ according to that order.
For $A\subseteq V^2$ denote $A_i(a)=\{(x_1,x_2)\in A\bigm| x_i=a\}$. 
We say that $A$ is {\em compressed} if 
$A_i(a) = \{0, 1, . . . , |A_i(a)|-1\}$ for $i = 1,2$ and any $a\in V$.
It is known (see, e.g. \cite{B2}) that if $G$ admits an optimal order then 
there exist compressed optimal sets of $V_G\times V_G$.
If $A\subseteq V^2$ is compressed then
\begin{equation}\label{e1}
  |I(A)| = \sum_{(x,y)\in A} (\delta_G(x) + \delta_G(y)).
\end{equation}
In this and next section we assume $V$ is the vertex set of $H_{s,p,i}$ and 
view the vertices of $V\times V$ as being placed in a matrix $M_{V^2}$. 
The columns and rows of this matrix represent the vertices of $V$ ordered 
in its optimal order from left to right and from bottom to top. 
A set of vertices of $M_{V^2}$ is called {\em downset} if the corresponding 
set of $V\times V$ is compressed.
 
We associate a weight with every vertex $(x,y)\in M_{V^2}$ 
defined by $w(x,y)=\delta_{H_{s,p,i}}(x)+\delta_{H_{s,p,i}}(y)$ and 
for a set $Z\subseteq M_{V^2}$
define $w(Z) = \sum_{(x,y)\in Z} w(x,y)$. This way we can consider a slightly 
more general problem than EIP: for a given $m$ find an $m$-element downset of
$M_{V^2}$ with maximum weight. It is easily seen that the weight of a maximum
weight downset is equal to $I_{H_{s,p,i}}(m)$ if the graph $H_{s,p,i}$ 
with the corresponding $\delta$-sequence exists. However, we abstract 
from existence of the graph and will be dealing with maximization of the 
function $I(\cdot)$ defined by (\ref{e1}) for compressed sets even if the 
graph $H_{s,p,i}$ does not exist, assuming in this case 
that we actually maximize the weight of a downset of $M_{V^2}$. This way we 
prove a minor extension of Theorem \ref{main} for maximum weight downsets 
of $M_{V^2}$.

For $A, B\subseteq V^2$ such that $B$ and $B\cup A$ are compressed and
$A\cap B = \emptyset$ denote
  $$d(A) = |I(B\cup A) \setminus I(B)|.$$
 
\begin{lm}\label{l1}
If $(x, y)\in V^2$ and $(x, y+qp)\in V^2$ then
    $$d(\lbrace (x, y + qp)\rbrace) - d(\lbrace (x, y)\rbrace) = q(p - i)$$
\end{lm}
\proof
Denote $t = \lfloor y/p\rfloor$. It can be easily shown that
\begin{eqnarray*}
	d(\lbrace (x, y)\rbrace ) &=& (y \bmod p) + t(p - i).\\
	d(\lbrace (x, y + qp)\rbrace) &=& ((y + qp) \bmod p) + (t + q)(p - i).
\end{eqnarray*}
Hence, we have
\begin{eqnarray*}
	d(\lbrace (x, y + qp)\rbrace) - d(\lbrace (x, y)\rbrace) 
	&=& ((y + qp) \bmod p) + (t + q)(p - i) - (y \bmod p) - t(p - i)\\
	&=& q(p - i).
\end{eqnarray*}
\qed

\begin{lm}\label{l2}
Let $A_1, A_2\subseteq V^2$ such that 
\begin{eqnarray*}
  A_1 &=& \{(x, y)\bigm |  x=a \mbox{ and } b\leq y < b+p \mbox{ for some }
    a, b\}\\
  A_2 &=& \{ (x, y + q)\bigm | (x, y)\in A_1\}
\end{eqnarray*}
one has $d(A_2) - d(A_1) = q(p - i)$.
\end{lm}
\proof
In other words, $A_1$ is a subset of consecutive vertices in a column of
$M_{V^2}$ and $A_2$ is a shift of $A_1$ on $q$ rows up.

We prove the lemma by induction on $q$. For $q=1$ Lemma \ref{l1} implies
\begin{eqnarray*}
	d(A_2) - d(A_1) 
	&=& d(\lbrace (x, y + p) \rbrace) - d(\lbrace (x, y) \rbrace)\\
	&=& p - i.
\end{eqnarray*}
Assume the statement holds for all $q'\leq q$. By induction we have
\begin{eqnarray*}
	d(A_2) - d(A_1) 
	&=& q(p - i) + d(\lbrace (x, y + q) \rbrace) - 
              d(\lbrace(x, y + q + p \rbrace)\\
	&=& q(p - i) + p - i\\
	&=& (q + 1)(p - i)
\end{eqnarray*}
\qed

\begin{lm}\label{l3}
Let $A_1, A_2 \subseteq V^2$ such that
\begin{eqnarray*}
  A_1 &=& \{ (x, y)\bigm | u\leq x < u + p \mbox{ for some } u\geq y 
         \mbox{ and fixed } y\}\\
  A_2 &=& \{ (y+r, x-d)\bigm | (x, y)\in A_1\mbox{ for some } 0\leq r < 
\lceil x/p\rceil p - x \}.
\end{eqnarray*}
Then $d(A_2) - d(A_1) = rp - d(p - i)$.
\end{lm}
\proof
The condition $u\geq y$ provides that $A_1$ is below the main diagonal of
$M_{V^2}$.
Let $R$ be the {\em reflections} of $A_1$ about the main diagonal and $S$ 
be its {\em shift down} on $d$ rows. Formally,
\begin{eqnarray*}
R &=& \lbrace (y, x)\bigm | (x, y)\in A_1\rbrace,\\
S &=& \lbrace (x, y-d)\bigm | (x, y)\in R\rbrace.
\end{eqnarray*}
Then $A_2$ is a {\em horizontal shift} of $S$ right on $r$ columns.
We have $d(R) - d(A_1) = 0$ and $d(R) - d(S)=d(p-i)$ by Lemma \ref{l2}. 
Moreover, $d(A_2) - d(S) = rp$. Therefore,
\begin{eqnarray*}
d(A_2) - d(A_1)
&=& d(A_2) - d(S) + d(S) - d(R)\\
&=& d(A_2) - d(S) - (d(R) - d(S))\\
&=& rp - d(p - i).
\end{eqnarray*}
\qed

\begin{lm}\label{l4}
For a fixed $a\geq 1$ and sets $A=\{(x,y)\bigm| x=a\}\subseteq V^2$ and 
$B=\{(x-1,y)\bigm |(x,y)\in A\}$ one has $d(B)-d(A)\geq -|B|=-|A|$.
\end{lm}
\proof
If $a\neq 0\bmod p$ then 
$$d(B) - d(A)=-|B|=-|A|.$$
Otherwise, if $a = wp$ for some $w\in \{1,\dots,s - 1\}$ then 
\begin{eqnarray*}
	d(B) - d(A)
	&=& |B|(p - 1 + (w - 1)(p - i)) 
	- |B|w(p - i)\\
	&=& |B|(p - 1 - p + i)\\
	&=& |B|(i - 1)\\
	&\geq & 0.
\end{eqnarray*}
So we have $d(B)-d(A)\geq -|B|=-|A|$ as desired.
\qed

\begin{tm}\label{tm2}
{\rm (Bezrukov, Els\"{a}sser \cite{BE})}
If $p \geq 3$ and $1 \leq i \leq p-i$ then the lexicographic order is optimal
for $H^2_{2,p,i}$.
\end{tm}

We use Theorem \ref{tm2} and some parts of its proof in the proof of our 
main result. 

\begin{lm}{\rm (see, e.g., \cite{BE})}\label{l5}
Let $G=(V_G,E_G)$ be a regular graph and $A\subseteq V_G$. Then $A$ is an 
optimal set iff $\overline{A}=V_G\setminus A$ is optimal.
\end{lm}

This lemma implies that that regardless which $i$ perfect 
matchings are removed from $K_{2p}$ to obtain $H_{2,p,i}$, the resulting
graph admits an optimal order. Note that Theorem \ref{tm2} 
in \cite{BE} deals with a specific removal of $i$ perfect matchings. However,
it remains valid for removal of any $i$ disjoint perfect matchings.

\section{Proof of Theorem \ref{main}}

Let $V=V_{H_{s,p,i}}$ and $A\subseteq V^2$ be a compressed set. We also 
assume that $A$ is stable under reflection about the main diagonal of
$M_{V^2}$. That is, the condition $(x,y)\in A$ implies $(y,x)\in A$.
For $1\leq k,l\leq s$ and $A\subseteq V^2$ denote 
\begin{eqnarray*}
	V_{k,l} &=& \{(u,v)\bigm| (k-1)p\leq u\leq kp-1,\; (l-1)p\leq v\leq lp-1\}\\
	A_{k,l} &=& A\cap V_{k,l}
\end{eqnarray*}

We prove the theorem by induction on $s$. The base case $s=2$ is proved in
\cite{BE}. Assume $s\geq 3$ and the theorem holds for all $s'<s$.

{\it Case 1.} Assume 
$A\subseteq \bigcup^{s-1}_{k=1} \left( \bigcup^s_{l=1} V_{k,l}\right)$. 
Let $A', A'' \subseteq A$ such that 
$A' = \bigcup^{s-1}_{k=1} \left( \bigcup^{s-1}_{l=1} V_{k,l}\right)$ and\\
$A'' = \bigcup^{s-1}_{k=1} \left( \bigcup^{s}_{l=2} V_{k,l}\right)$.
We transform $A$ by replacing $A'$ with $\F_{A'}^2(|A'|)$ and $A''$ with 
$\F_{A''}^2(|A''|)$. It is easily seen that the resulting set $B$ will be 
compressed 
and optimal. Moreover, the sum of the lexicographic numbers of vertices of 
$B$ is less than the one for $A$. Therefore, after a finite number of 
repetitions of such transformation we obtain a set $B$ which is stable under
this transformation. For $B$ denote
\begin{eqnarray*}
	k_1 &=& \max_{(q, 0)\in B} q,\\
	k_2 &=& \max_{(q, sp - 1)\in B} q + 1,\\
	y_1 &=& \max_{(k_1, q)\in B} q,\\
	y_2 &=& \max_{(k_2, q)\in B} q.
\end{eqnarray*}
Taking into account that $B$ is compressed, we have $k_2 \leq k_1\leq k_2+1$.
If $k_1 = k_2$ then $y_1 = y_2$, hence $B={\cal F}_V^2(|B|)$. Therefore,
without loss of generality we assume $k_1=k_2+1$ and $y_1 < p$. We have 
$y_2 \bmod p\geq y_1$. Indeed, if this is not the case then we transform 
$B$ to $C$ by replacing $B'=\{(k_1,y)\bigm |y_2\bmod p<y\leq y_1\}\subseteq B$ 
with $C'=\{(k_2,y)\bigm |y_2<y\leq y_1+(s-1)p\}\subseteq V^2\setminus B$. 
By taking into account Lemma \ref{l4} and Lemma \ref{l1} we have 
$|I(C)|-|I(B)|\geq |B'|(s-1)(p-i)-|B'|=|B'|((s-1)(p-i)-1)>0$, which contradicts 
the optimality of $B$.

If $y_2 \bmod p\geq y_1$ and $sp-y_2-1> y_1+1$ we transform $B$ to $D$ by 
replacing $B'=\{(k_1,y)\bigm |0\leq y\leq y_1\}\subseteq B$ with 
$D'=\{(k_2,y)\bigm |y_2<y\leq y_2+y_1+1\}\subseteq V^2\setminus B$. 
By Lemmas \ref{l4} and \ref{l1} we have 
$|I(D)|-|I(B)|> |B'|(s-1)(p-i)-|B'|=|B'|((s-1)(p-i)-1)>0$, which again 
contradicts the optimality of $B$.

\begin{figure}[htb]
\begin{center}\unitlength=0.8mm \begin{picture}(105,105)\thicklines
\definecolor{light}{rgb}{.9,.9,.9}
\definecolor{dark}{rgb}{.7,.7,.7}
\put(0,1){\colorbox{light}{\makebox(45,102.5)[ss]{}} }
\put(47,1){\colorbox{light}{\makebox(2.5,90)[ss]{}} }
\put(52,1.5){\colorbox{dark}{\makebox(2.5,5)[ss]{}} }
\put(0,0){\line(1,0){105}} \put(0,0){\line(0,1){105}}
\put(105,105){\line(-1,0){105}} \put(105,105){\line(0,-1){105}}
{\thinlines
\put(35,0){\line(0,1){105}} \put(70,0){\line(0,1){105}}
\put(0,35){\line(1,0){105}} \put(0,70){\line(1,0){105}}
\put(52,92.5){\line(0,1){6.5}}
\put(47.5,92.5){\line(0,1){6.5}}
\put(47.5,99){\line(1,0){4.5}}
\put(47.5,92.5){\line(1,0){4.5}}

\put(55,5){\line(0,1){91}}
\put(55,96){\vector(-1,0){6}}
}
\put(60,8){\makebox(0,0)[cc]{$y_1$}}
\put(60,93){\makebox(0,0)[cc]{$y_2$}}
\end{picture}
\caption{Transforming set $B$ into $D$ in {\it Case 1}}\label{f2}
\end{center}
\end{figure}
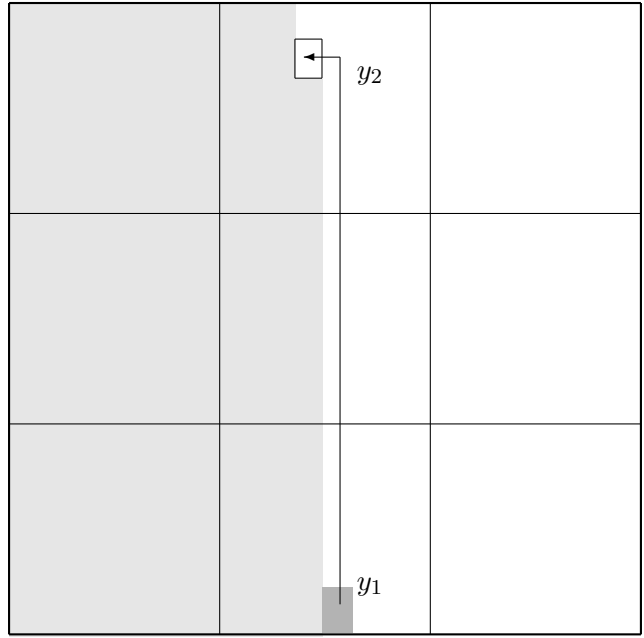

If $y_2 \bmod p\geq y_1$ and $sp-y_2-1\leq y_1+1$ we transform $B$ to $H$ by 
replacing $B'=\{(k_1,y)\bigm |y_1-(sp-y_2-1)< y\leq y_1\}$ with 
$H'=\{(k_2,y)\bigm |y_2<y<sp\}$. By Lemmas \ref{l4} and \ref{l1} we have 
$|I(H)|-|I(B)|> |B'|(s-1)(p-i)-|B'|=|B'|((s-1)(p-i)-1)>0$, which still 
contradicts the optimality of $B$.

{\it Case 2.} Assume 
$A \subseteq V^2$ and $A_{s,1}\neq\emptyset$
Denote
\begin{eqnarray*}
V_h &=& \bigcup_{k=1}^{s-1} ( \bigcup_{l=2}^s V_{k,l}) ,\quad  
   A_h =A\cap V_h\\
V_2 &=& \bigcup_{l=2}^s V_{1,l} ,\quad  A_2 = A\cap V_2\\
V_3 &=& \bigcup_{k=2}^s V_{k,1} ,\quad  A_3 = A\cap V_3.
\end{eqnarray*}

By inductive hypothesis we can assume that $A_h$ is an initial segment 
in ``columns'' of $V_h$, and also $A_3$ is an in initial segment in ``rows'' 
of $V_3$. Without loss of generality we can assume that $A_h = A_2$, 
since otherwise we can apply arguments similar to the one of {\it Case 1} 
to further transform the set to $\F_V^2(|A|)$. After given transformation 
the number of inner edges between $A_{1,1}$ and $A_2$ and between 
$A_{1,1}$ and $A_3$ 
does not decrease. Let 
\begin{eqnarray*}
  |A_2| &=& k_2(s-1)p + \gamma_2\\
  |A_3| &=& k_3(s-1)p + \gamma_3
\end{eqnarray*}
with $0 \leq \gamma_2,\gamma_3 < (s-1)p$. From the arguments above $k_2<p$.
Since $A$ is stable under reflection, we have $k_2\geq k_3$.

Let $\epsilon = 0$ if $\gamma_3 > 0$ and $\epsilon = 1$ if $\gamma_3 = 0$. 
Let $V_5 = \{(x,y) \in V_{1,1} | k_2 \leq x < p, k_3 + \epsilon \leq y < p \}$. 
Replace $A_5 = A \cap V_5$ with the initial segment of $V_5$. It is easily
seen that the resulting set remains optimal.
Let $|A_5|=(p-k_3-\epsilon)k_5 + \gamma_5$ with $0\leq\gamma_5 
< p-k_3-\epsilon +1$.

\begin{figure}[H]
\begin{center}\unitlength=0.8mm \begin{picture}(105,105)\thicklines
\definecolor{light}{rgb}{.9,.9,.9}
\definecolor{dark}{rgb}{.7,.7,.7}
\put(0,1){\colorbox{light}{\makebox(9,102.5)[ss]{}} }
\put(11,1){\colorbox{light}{\makebox(2.4,80)[ss]{}} }
\put(16,16){\colorbox{dark}{\makebox(9,18)[ss]{}} }
\put(26,16){\colorbox{dark}{\makebox(3,10)[ss]{}} }
\put(0,1){\colorbox{light}{\makebox(102.5,9)[ss]{}} }
\put(0,11){\colorbox{light}{\makebox(50,2.5)[ss]{}} }
\put(0,0){\line(1,0){105}} \put(0,0){\line(0,1){105}}
\put(105,105){\line(-1,0){105}} \put(105,105){\line(0,-1){105}}
{\thinlines
\put(35,0){\line(0,1){105}} \put(70,0){\line(0,1){105}}
\put(0,35){\line(1,0){105}} \put(0,70){\line(1,0){105}}
\put(16,15){\line(0,1){20}} \put(16,15){\line(1,0){19}} }
\put(0,95){\vector(1,0){11}} \put(11,95){\vector(-1,0){11}}
\put(16,30){\vector(1,0){11}} \put(27,30){\vector(-1,0){11}}
\put(35,13){\vector(1,0){17}} \put(52,13){\vector(-1,0){17}}
\put(13,35){\vector(0,1){47}} \put(13,82){\vector(0,-1){47}}
\put(30,15){\vector(0,1){12}} \put(30,27){\vector(0,-1){12}}
\put(90,0){\vector(0,1){11}} \put(90,11){\vector(0,-1){11}}
\put(6,98){\makebox(0,0)[cc]{$k_2$}}
\put(21,33){\makebox(0,0)[cc]{$k_5$}}
\put(95,6){\makebox(0,0)[cc]{$k_3$}}
\put(10,55){\makebox(0,0)[cc]{$\gamma_2$}}
\put(43,10){\makebox(0,0)[cc]{$\gamma_3$}}
\put(27,20){\makebox(0,0)[cc]{$\gamma_5$}}
\end{picture}
\caption{Notations used for {\it Case 2}}\label{f3}
\end{center}
\end{figure}
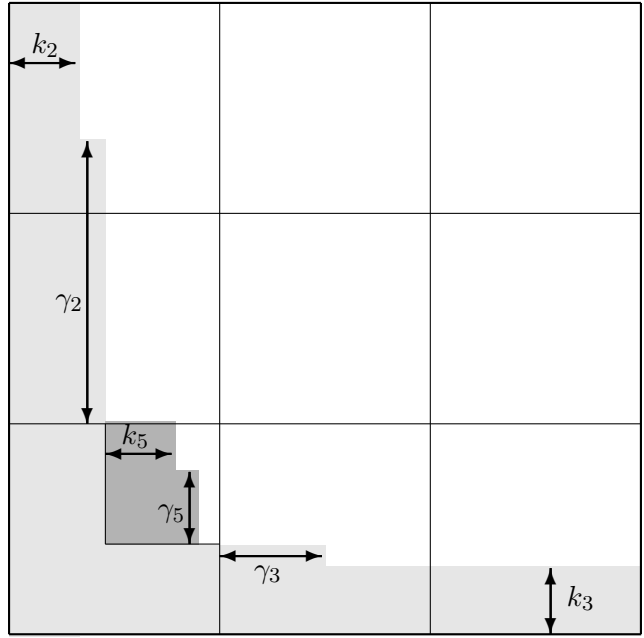

{\it Case 2a.}
We show that $\gamma_2 = 0$.

If $k_3 > 0$ and $\gamma_3 > 0$, we first exchange the sets $A'=\{(x,0)\bigm|
p+ \gamma_3 \leq x < sp\}$ and $B'=\{(x,k_3)\bigm|(x,0)\in A'\}$ and
transform $A$ to a resulting set $B$ such that $|I(B)|=|I(A)|$. 
If $k_3 > 0$ and $\gamma_3 = 0$, we transform $B$ to a resulting set $C$ by 
exchanging $B'=\{(x,k_3-\epsilon)\bigm| p+\gamma_2 \leq x < sp\}$ with 
$C'=\{(k_2,y)\bigm | (y,k_3-\epsilon)\in B'\}$. Since $B'$ and $C'$ are
isomorphic segments in $V$ and  $k_2\geq k_3$ we have that $|I(C)|-|I(B)|>0$. 
Hence we have a contradiction with the optimality of $C$.

If $k_3 = 0$  we have that $\gamma_2 \geq \gamma_3$. If it is not the case
then we transform $A$ to $D$ by exchanging $A'=\{(x,0)\bigm | p+\gamma_2 \leq
x < p + \gamma_3\}$ with $D'=\{(k_2,y)\bigm | (y,0)\in A'\}$. Since $A'$ and
$D'$ represent isomorphic segments in $V$ and  $k_2\geq k_3$ we have that $|I(D)|
- |I(A)| \geq 0$. Hence the set $D$ is optimal. So we can assume $\gamma_2 \geq \gamma_3$.

If $k_3 = 0$, $\gamma_2 \geq \gamma_3$, $\gamma_3>(s-2)p$ and $\gamma_3
-(s-2)p\leq (s-1)p -\gamma_2$ we transform $A$ to resulting set $H$ by
replacing $A'=\{(0,b)\bigm |(s-1)p\leq a<\gamma_3\}$ with $H'=\{(b,k_2)\bigm
|\gamma_2 \leq b<p+\gamma_2 +\gamma_3 -(s-2)p\}$. We have 
$|I(H)|-|I(A)|>k_2(\gamma_3-(s-2)p)>0$. Now we can apply arguments of 
{\it Case 1} on H.

If $k_3 = 0$, $\gamma_2 \geq \gamma_3$, $\gamma_3>(s-2)p$ and $\gamma_3 -(s-2)p >(s-1)p -\gamma_2$ we transform $A$ to resulting set $K$ by replacing $A'=\{(0,b)\bigm |\gamma_3 -((s-1)p -\gamma_2
) \leq a<\gamma_3\}$ with $K'=\{(b,k_2)\bigm | \gamma_2 \leq b< sp\}$. We have
  $$|I(K)|-|I(A)| > ((s-1)p - \gamma_2)k_2 \geq 0.$$
Hence, we have a contradiction with the optimality of $A$, which implies
$\gamma_2 =0$.

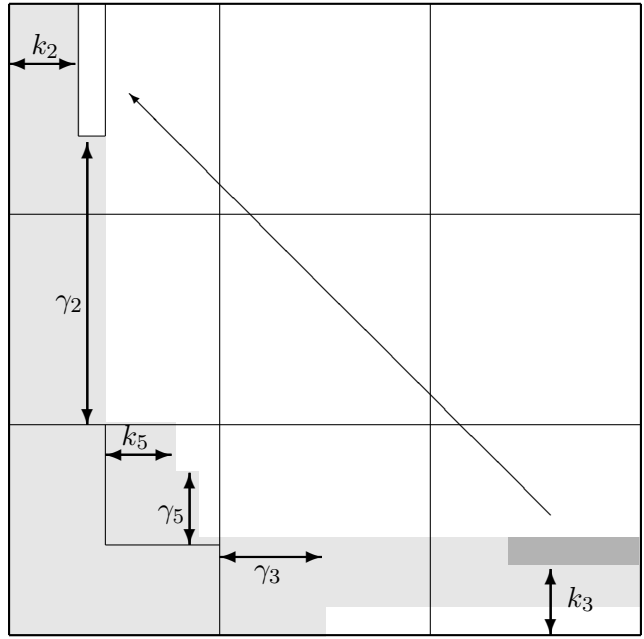
\begin{figure}[H]
\begin{center}\unitlength=0.8mm \begin{picture}(105,105)\thicklines
\definecolor{light}{rgb}{.9,.9,.9}
\definecolor{dark}{rgb}{.7,.7,.7}
\put(0,1){\colorbox{light}{\makebox(9,102.5)[ss]{}} }
\put(11,1){\colorbox{light}{\makebox(2.4,80.5)[ss]{}} }
\put(16,16){\colorbox{light}{\makebox(9,18)[ss]{}} }
\put(26,16){\colorbox{light}{\makebox(3,10)[ss]{}} }
\put(0,6){\colorbox{light}{\makebox(102.5,9)[ss]{}} }
\put(0,1){\colorbox{light}{\makebox(50,2.5)[ss]{}} }
\put(83,13){\colorbox{dark}{\makebox(19,2)[ss]{}} }
{\thinlines
\put(11.5,83){\line(0,1){22}} \put(11.5,83){\line(1,0){4.5}}
\put(16,83){\line(0,1){22}} \put(90,20){\vector(-1,1){70}}
}
\put(0,0){\line(1,0){105}} \put(0,0){\line(0,1){105}}
\put(105,105){\line(-1,0){105}} \put(105,105){\line(0,-1){105}}
{\thinlines
\put(35,0){\line(0,1){105}} \put(70,0){\line(0,1){105}}
\put(0,35){\line(1,0){105}} \put(0,70){\line(1,0){105}}
\put(16,15){\line(0,1){20}} \put(16,15){\line(1,0){19}} }

\put(0,95){\vector(1,0){11}} \put(11,95){\vector(-1,0){11}}
\put(16,30){\vector(1,0){11}} \put(27,30){\vector(-1,0){11}}
\put(35,13){\vector(1,0){17}} \put(52,13){\vector(-1,0){17}}
\put(13,35){\vector(0,1){47}} \put(13,82){\vector(0,-1){47}}
\put(30,15){\vector(0,1){12}} \put(30,27){\vector(0,-1){12}}
\put(90,0){\vector(0,1){11}} \put(90,11){\vector(0,-1){11}}
\put(6,98){\makebox(0,0)[cc]{$k_2$}}
\put(21,33){\makebox(0,0)[cc]{$k_5$}}
\put(95,6){\makebox(0,0)[cc]{$k_3$}}
\put(10,55){\makebox(0,0)[cc]{$\gamma_2$}}
\put(43,10){\makebox(0,0)[cc]{$\gamma_3$}}
\put(27,20){\makebox(0,0)[cc]{$\gamma_5$}}
\end{picture}
\caption{Transforming set $A$ to $C$ in {\it Case 2a} for $k_3>0$}\label{f4}
\end{center}
\end{figure}

{\it Case 2b.}
Assume $\gamma_2 = 0$ and $k_5 > 0$. If $k_5>k_3$ or $k_5 = k_3$ and $\gamma_3 = 0$, then 
$|A| \leq sp^2$. In this case we move a point $(x,y)$ to $(y+k_2,x)$ for
every $(x,y) \in A_3$. Such transformation results in a set $C\subseteq V_2$
and it is easily seen that $|I(C)| \geq |I(A)|$. Moreover, $C$ satisfies 
conditions of {\it Case 1}. So we use arguments of {\it Case 1} to further 
transform $C$ to $\F_V^2(|A|)$. 

If $1 \leq k_5 < k_3 $ or $ 1 \leq k_5 = k_3 $ and $ \gamma_3 >0$, we move 
$(x,y) \in A_3 $ to $(y +k_2,x) \in V_2\setminus A$ for $p\leq x< sp$ and 
$0\leq y <k_5$. We show that for the resulting set $C$ we have
$|I(C)|-|I(A)|> 0$.

We move $(x,y) \in A_3$ for $(s-1)p \leq x <sp$ and $ 0 \leq y <k_5$ to 
$(y+k_2,x) \in V_2\setminus A$.
For the resulting set $C$, we have,
  $$|I(C)|-|I(A)| = k_5(s-1)p(k_2-(k_3-k_5))-\gamma_3 k_5 \geq 
k_5((s-1)p-\gamma_3)>0$$
which contradicts the optimality of $A$.

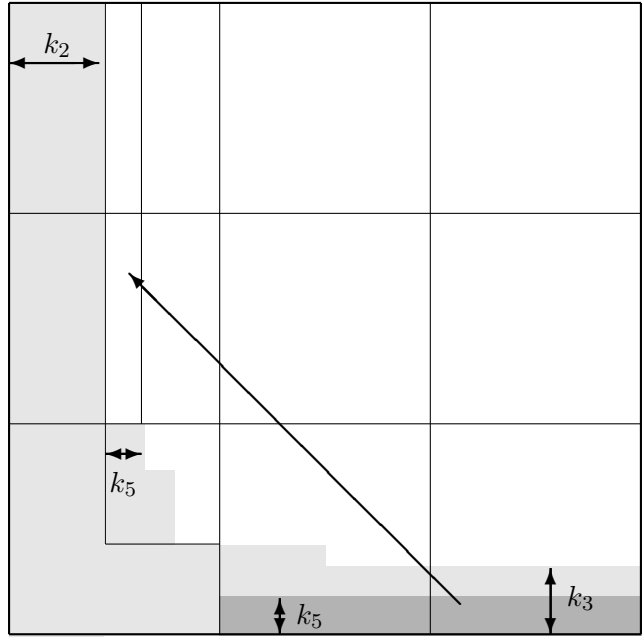
\begin{figure}[H]
\begin{center}\unitlength=0.8mm \begin{picture}(105,105)\thicklines
\definecolor{light}{rgb}{.9,.9,.9}
\definecolor{dark}{rgb}{.7,.7,.7}
\put(0,1){\colorbox{light}{\makebox(13,102.5)[ss]{}} }
\put(16,16){\colorbox{light}{\makebox(4,17.5)[ss]{}} }
\put(22,16){\colorbox{light}{\makebox(3,10)[ss]{}} }
\put(0,1){\colorbox{light}{\makebox(102.5,9)[ss]{}} }
\put(0,11){\colorbox{light}{\makebox(50,2.5)[ss]{}} }
\put(35,1){\colorbox{dark}{\makebox(67.5,4)[ss]{}} }
\put(0,0){\line(1,0){105}} \put(0,0){\line(0,1){105}}
\put(105,105){\line(-1,0){105}} \put(105,105){\line(0,-1){105}}
{\thinlines
\put(35,0){\line(0,1){105}} \put(70,0){\line(0,1){105}}
\put(0,35){\line(1,0){105}} \put(0,70){\line(1,0){105}}
\put(16,15){\line(0,1){20}} \put(16,15){\line(1,0){19}}
\put(16,35){\line(0,1){70}}
\put(22,35){\line(0,1){70}}
}
\put(0,95){\vector(1,0){15}} \put(11,95){\vector(-1,0){11}}
\put(90,0){\vector(0,1){11}} \put(90,11){\vector(0,-1){11}}
\put(45,0){\vector(0,1){6}} \put(45,6){\vector(0,-1){6}}
\put(16,30){\vector(1,0){6}} \put(22,30){\vector(-1,0){6}}
\put(8,98){\makebox(0,0)[cc]{$k_2$}}
\put(95,6){\makebox(0,0)[cc]{$k_3$}}
\put(50,3){\makebox(0,0)[cc]{$k_5$}}
\put(19,25){\makebox(0,0)[cc]{$k_5$}}
\put(75,5){\vector(-1,1){55}}
\end{picture}
\caption{Transforming set $A$ into $C$ in {\it Case 2b} for $k_5 <
k_3$}\label{f5}
\end{center}
\end{figure}

{\it Case 2c.}
Assume $\gamma_2 = 0$, $k_5 = 0$, $k_3 > 0$ and $k_2 < p-1$.

We first transform $A$ to $C$ by replacing the set
$C'= \{(k_2, y)\bigm | k_3 \leq y < k_3 + \gamma_5\}$ with
$\{(x+1, y)\bigm | (x, y) \in C'\}$ and if $\gamma_3>0$, replacing the set  
$C''= \{(x, 0)\bigm | p + \gamma_3 \leq x < sp \}$ with
$\{(x, k_3)\bigm | (x, 0)\in C''\}$. For the resulting set $C$ one has
$|I(A)|=|I(C)|$. Denote
\begin{eqnarray*}
P &=& \{(x, k_3 - \epsilon)\bigm | k_2 \leq x < sp \},\\
Q &=& \{(k_2, y)\bigm |k_3-\epsilon \leq y< k_3-\epsilon +sp-				k_2\} \\
  &=& \{(k_2,y)\bigm |k_2-(k_2-k_3+\epsilon) \leq y <sp-(k_2-				k_3+\epsilon) \}\\
  &=& \{(x+(k_2-k_3+\epsilon),y-(k_2-k_3+\epsilon)\bigm|(y,x)\in P\}
\end{eqnarray*} 
Now we further transform $C$ into $B$ by replacing $P$ with $Q$. 
We prove by induction on $s$ that $|I(B)| - |I(C)|>0$. For $s = 2$ this is 
proved in \cite{BE}. Assume $s \geq 3$ and this is true for all $s' < s$.

Since $|P\cap V_{s,1}| = p$ and
$s\geq 3$ we have $|P| = |Q| > p$. Let $P'\subseteq P$ be the set of $p$ 
vertices of $P$ with largest $x$-coordinates. Similarly, let $Q'\subseteq Q$ 
be the set of $p$ vertices of $Q$ with largest $y$-coordinates. Since
$|P'| = |Q'| = p$, the $x$-coordinates of vertices of $P\setminus P'$ and
$y$-coordinates of vertices of $Q\setminus Q'$ do not exceed $(s-1)p$. 
By induction on $s$ we have $d(Q\setminus Q') - d(P\setminus P') > 0$.
Lemma \ref{l3} implies 
$d(Q')-d(P')=(k_2-k_3+\epsilon)p-(k_2-k_3+\epsilon)(p - i)$.
Hence,
\begin{eqnarray*}
	|I(B)| - |I(C)| &\geq& d(Q) - d(P) \\
	&=& d(Q\setminus Q') + d(Q') - (d(P\setminus P') + d(P'))\\
	&=& d(Q\setminus Q') - d((P\setminus P')) + (k_2-k_3+\epsilon)p- (k_2-k_3+\epsilon)(p-i)\\ 
	&>& (k_2-k_3+\epsilon)p-(k_2-k_3+\epsilon)(p-i)\\ 
	&=& (k_2-k_3+\epsilon)i \\
	&\geq & 0
\end{eqnarray*}

\begin{figure}[H]
\begin{center}\unitlength=0.8mm \begin{picture}(105,105)\thicklines
\definecolor{light}{rgb}{.9,.9,.9}
\definecolor{dark}{rgb}{.7,.7,.7}
\put(0,1){\colorbox{light}{\makebox(13,102.5)[ss]{}} }
\put(22,16){\colorbox{light}{\makebox(3,10)[ss]{}} }
\put(0,6){\colorbox{light}{\makebox(102.5,3)[ss]{}} }
\put(15,1){\colorbox{light}{\makebox(45,3)[ss]{}} }
\put(15.5,11.5){\colorbox{dark}{\makebox(87,2)[ss]{}} }
\put(0,0){\line(1,0){105}} \put(0,0){\line(0,1){105}}
\put(105,105){\line(-1,0){105}} \put(105,105){\line(0,-1){105}}
{\thinlines
\put(35,0){\line(0,1){105}} \put(70,0){\line(0,1){105}}
\put(0,35){\line(1,0){105}} \put(0,70){\line(1,0){105}}
\put(16,10.5){\line(0,1){25}}
\put(16,35){\line(0,1){70}}
\put(22,10.5){\line(0,1){94.5}}
}
\put(0,95){\vector(1,0){15}} \put(11,95){\vector(-1,0){11}}
\put(90,0){\vector(0,1){11}} \put(90,11){\vector(0,-1){11}}
\put(8,98){\makebox(0,0)[cc]{$k_2$}}
\put(95,6){\makebox(0,0)[cc]{$k_3$}}

\put(48,5){\makebox(0,0)[cc]{$\gamma_3$}}
\put(47,3){\vector(1,0){16}} \put(51,3){\vector(-1,0){16}}
\put(30,20){\makebox(0,0)[cc]{$\gamma_5$}}
\put(25,16){\vector(0,1){11}} \put(25,26){\vector(0,-1){11}}

\put(75,12){\vector(-1,1){55}}
\end{picture}
\caption{Transforming set $C$ into $B$ in {\it Case 2c}}\label{f6}
\end{center}
\end{figure}
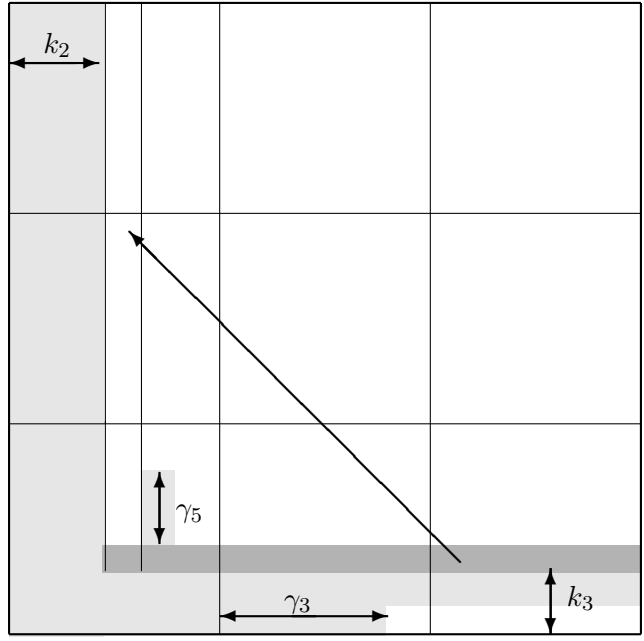

{\it Case 2d.}
Assume $\gamma_2 = 0$, $k_5 = 0$, $k_3 > 0$ and $k_2 = p-1$.
If $\gamma_3 > 0$ we first transform $A$ to $C$ by replacing the set  
$C' = \{(x, 0)\bigm | p + \gamma_3 \leq x < sp \}$ with
$\{(x, k_3)\bigm | (x, 0)\in C'\}$. For the resulting set $C$ one has
$|I(A)|=|I(C)|$, so $C$ is optimal. Denote
\begin{eqnarray*}
P &=& \{(x, k_3 - \epsilon)\bigm | p \leq x < sp \},\\
Q &=& \{(p-1,y)\bigm | k_3-\epsilon +\gamma_5 + 1 \leq y<sp-p+k_3-					\epsilon +\gamma_5 + 1\}\\
  &=& \{(p-1,y)\bigm | p-(p-1-k_3+\epsilon -\gamma_5)\leq y<sp-(p-1-				k_3+\epsilon -\gamma_5)\}\\
  &=& \{(x+(p-1-k_3 + \epsilon),y-(p-1-k_3+\epsilon -\gamma_5)\bigm |				(y,x)\in P\}.
\end{eqnarray*}
Let $B=(C\setminus P)\cup Q$. We prove by induction on $s$ that $|I(B)| -
|I(C)|>0$. For $s = 2$ this is proved in \cite{BE}. Assume $s \geq 3$ and
this is true for all $s' < s$.

Since $|P\cap V_{s,1}| = p$ and
$s\geq 3$ we have $|P| = |Q| > p$. Let $P'\subseteq P$ be the set of $p$ 
vertices of $P$ with largest $x$-coordinates. Similarly, let $Q'\subseteq Q$ 
be the set of $p$ vertices of $Q$ with largest $y$-coordinates. Since
$|P'| = |Q'| = p$, the $x$-coordinates of vertices of $P\setminus P'$ and
$y$-coordinates of vertices of $Q\setminus Q'$ do not exceed $(s-1)p$. 
By induction on $s$ we have $d(Q\setminus Q') - d(P\setminus P') > 0$.
Lemma \ref{l3} implies
$d(Q')-d(P')=(p-1-k_3 + \epsilon)p-(p-1-k_3+\epsilon -\gamma_5)(p - i)$. 
Hence,
\begin{eqnarray*}
	|I(B)| - |I(C)| &\geq& d(Q) - d(P) \\
	&=& d(Q\setminus Q') + d(Q') - (d(P\setminus P') + d(P'))\\
	&=& d(Q\setminus Q') - d((P\setminus P')) + (p-1-k_3 + \epsilon)p-(p-1-k_3+\epsilon -\gamma_5)(p - i)\\ 
	&>& (p-1-k_3 + \epsilon)p-(p-1-k_3+\epsilon -\gamma_5)(p - i)\\ 
	&>& \gamma_5(p-i) \\
	&\geq & 0.
\end{eqnarray*}
Therefore, $|I(B)| - |I(C)|>0$ for all $s$, which, in turn, contradicts 
the optimality of $A$.

{\it Case 2e.}
Assume $\gamma_2 = 0$, $k_5 = 0$ and $k_3 = 0$.
We can assume that $\gamma_3 > (s-2)p$, since if this is not the case we 
can use arguments of {\it Case 1} to transform $A$ into $\F_V^2(|A|)$. 

If $\gamma_5 \geq k_2$ then we transform $A$ into a set $B$ by replacing
$B'=\{(x,0)\bigm | \gamma_5 + 1 \leq x < p+\gamma_3\}$ with
$B''=\{k_2,y)\bigm | (y,0)\in B'\}$. Since $B'$ and $B''$ represent
isomorphic segment in $V$ and  $k_2\geq k_3$ we have that 
$|I(B)| - |I(A)| \geq 0$. 
Now we can use arguments of {\it Case 1} to transform A into $\F_V^2(|A|)$.

{\it Case 2f.}
Assume $\gamma_2 = 0$, $k_5 = 0$, $k_3 = 0$, 
$\gamma_3 >(s-2)p$ and $\gamma_5 < k_2$.
Denote
\begin{eqnarray*}
P &=& \{(x,0)\bigm | k_2+1\leq x < p+\gamma_3\}\\
Q &=& \{(k_2,y)\bigm | \gamma_5 + 1\leq y<\gamma_5 +1+p+\gamma_3 -k_2 		-1\}\\
  &=& \{(k_2,y)\bigm | k_2+1-(k_2-\gamma_5)\leq y<p+\gamma_3-(k_2-				\gamma_5)\}\\
  &=& \{(x+k_2,y-(k_2-\gamma_5))\bigm | (y,x)\in P\}.
\end{eqnarray*}
Let $B = (A\setminus P)\cup Q$. We prove by induction on $s$ that
$|I(B)| - |I(A)|>0$. For $s = 2$ this is proved in \cite{BE}.
Assume $s \geq 3$ and this is true for all $s' < s$.

Since $\gamma_3 >(s-2)p$ and
$s\geq 3$ we have $|P| = |Q| > p$. Let $P'\subseteq P$ be the set of $p$ 
vertices of $P$ with largest $x$-coordinates. Similarly, let $Q'\subseteq Q$ 
be the set of $p$ vertices of $Q$ with largest $y$-coordinates. Since
$|P'| = |Q'| = p$, the $x$-coordinates of vertices of $P\setminus P'$ and
$y$-coordinates of vertices of $Q\setminus Q'$ do not exceed $(s-1)p$. 
By induction on $s$ we have $d(Q\setminus Q') - d(P\setminus P') > 0$.
Lemma \ref{l3} implies $d(Q')-d(P')=k_2p-(k_2-\gamma_5)(p - i)$.
One has
\begin{eqnarray*}
	|I(B)| - |I(A)| &\geq& d(Q) - d(P) \\
	&=& d(Q\setminus Q') + d(Q') - (d(P\setminus P') + d(P'))\\
	&=& d(Q\setminus Q') - d((P\setminus P')) + k_2p- (k_2-\gamma_5)			(p-i)\\ 
	&>& k_2p-(k_2-\gamma_5)(p-i)\\ 
	&>& \gamma_5 (p-i) \\
	&\geq & 0
\end{eqnarray*}

\begin{figure}[H]
\begin{center}\unitlength=0.8mm \begin{picture}(105,105)\thicklines
\definecolor{light}{rgb}{.9,.9,.9}
\definecolor{dark}{rgb}{.7,.7,.7}
\put(0,1){\colorbox{light}{\makebox(13,102.5)[ss]{}} }
\put(21,1.2){\colorbox{dark}{\makebox(70,2.5)[ss]{}} }
\put(21,1.2){\colorbox{dark}{\makebox(70,2.5)[ss]{}} }
\put(15,1){\colorbox{light}{\makebox(4,20)[ss]{}} }
\put(0,0){\line(1,0){105}} \put(0,0){\line(0,1){105}}
\put(105,105){\line(-1,0){105}} \put(105,105){\line(0,-1){105}}
{\thinlines
\put(35,0){\line(0,1){105}} \put(70,0){\line(0,1){105}}
\put(0,35){\line(1,0){105}} \put(0,70){\line(1,0){105}}
\put(15.5,22){\line(0,1){70}}
\put(21,22){\line(0,1){70}}
\put(15.5,92.2){\line(1,0){5.7}}
\put(65,4){\vector(-1,1){43}}
\put(75,8){\makebox(0,0)[cc]{$\gamma_3$}}
\put(63,3){\vector(1,0){30}} \put(65,3){\vector(-1,0){30}}
\put(24,13){\makebox(0,0)[cc]{$\gamma_5$}}
\put(20,11){\vector(0,1){11}} \put(20,17){\vector(0,-1){11}}
\put(0,95){\vector(1,0){15}} \put(11,95){\vector(-1,0){11}}
\put(8,98){\makebox(0,0)[cc]{$k_2$}} }
\end{picture}
\caption{Transforming set $A$ into $B$ in {\it Case 2f}}\label{f7}
\end{center}
\end{figure}
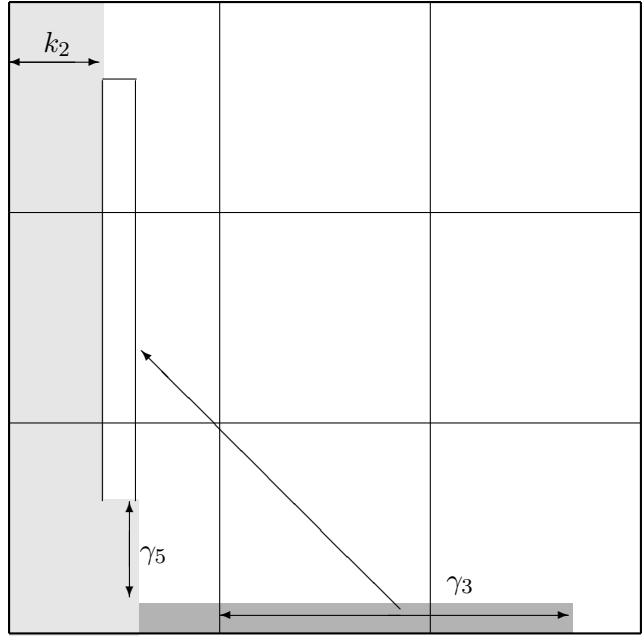

Hence, $|I(B)| - |I(A)|>0$ and we have a contradiction with the optimality
of $A$ for all $s$.

{\it Case 3.}
Assume $A\subset V^2$ and $A_{2,s}\neq \emptyset$. Hence, 
$A_{s,2}\neq\emptyset$. Let $\overline{A}=V^2\setminus A$. We transform 
$\overline{A}$ by replacing every $(x,y)\in A$ with $(sp-x, sp-y)$. 
By Lemma \ref{l5} the obtained set $B$ is optimal iff $A$ is optimal.
Moreover, $B$ is compressed and satisfies conditions of 
{\it Case 1} or {\it Case 2}.
Hence we can apply arguments from those cases to transform $B$ to 
$\F_V^2(|B|)$ without decreasing the number of its inner edges.
\qed

\section{Further results and concluding remarks}

Counterexamples show that not for every $\delta$-sequence there is a
corresponding graph $(V,E)$ \cite{B2}. A necessary (but not sufficient) 
condition for the graph existence is $\delta(i+1)\leq \delta(i)+1$ for 
$1\leq i < |V|$. For a graph to be connected it must hold $\delta(i)>0$ for
$i>1$. We call $\delta$-sequences satisfying these two conditions 
{\em appropriate}.

How typical is that the lexicographic order is optimal for the products of
regular graphs? To answer this question we generated all appropriate symmetric
$\delta$-sequences of small length $|V|$ and verified the corresponding
graphs $G$ (if exist) for the optimality on $G\times G$ \cite{D}. We call 
$\delta_G$ {\em isoperimetric} if $G\times G$ admits some optimal order.

It turns out that the lexicographic order is optimal for all of the explored 
$\delta$-sequences with an exclusion of the Petersen graph. At the same 
time we identified new graphs that admit optimal orders for all their 
Cartesian powers (due to the local-global principle \cite{AC}).
Graphs marked with an asterisk in the tables below were not previously studied.
There are no any new graphs for $n\leq 8$.

For $n=9$ there are $10$ symmetric appropriate $\delta$-sequences, out of 
which only 5 are isoperimetric.

\begin{table}[htb]
\begin{center}
\begin{tabular}{l|l}
$\delta$-sequence & graph \\\hline
$(0,1,1,2,2,2,3,3,4)$ & interesting new graph$^*$\\
$(0,1,2,1,2,3,2,3,4)$ & $K_3\times K_3$  or $K_9 - 2C_9^*$ or
                        $K_9 - (K_3\times K_3)^*$\\
$(0,1,2,2,3,4,4,5,6)$ & $K_{3,3,3}$  or $K_9 - 3C_3^*$ \\
$(0,1,2,3,3,3,4,5,6)$ & $K_9 - C_9^*$ \\
$(0,1,2,3,4,5,6,7,8)$ & $K_9$ \\
\end{tabular}
\caption{Symmetric isoperimetric sequences of length 9}
\end{center}
\end{table}

For $n=10$ there are $36$ symmetric $\delta$-sequences, out of which only 11
are isoperimetric. However, only one of them leads to a graph which was not
studied before.

\begin{table}[htb]
\begin{center}
\begin{tabular}{l|l}
$\delta$-sequence & graph \\\hline
$(0,1,1,1,2,1,2,2,2,3)$ & Petersen graph\\
$(0,1,1,2,1,2,1,2,2,3)$ & $C_5\times P_1$ \\
$(0,1,1,2,2,2,2,3,3,4)$ & $K_{5,5} - M$ \\
$(0,1,1,2,2,3,3,4,4,5)$ & $K_{5,5}$ \\
$(0,1,2,2,2,3,3,3,4,5)$ & $K_{10} - 4M$ \\
$(0,1,2,2,3,3,4,4,5,6)$ & $K_{10} - 3M$ \\
$(0,1,2,3,3,4,4,5,6,7)$ & $K_{10} - 2C_5^*$ \\
$(0,1,2,3,4,1,2,3,4,5)$ & $K_5 \times K_1$ \\
$(0,1,2,3,4,3,4,5,6,7)$ & $K_{10} - 2M$ \\
$(0,1,2,3,4,4,5,6,7,8)$ & $K_{10} - M$ \\
$(0,1,2,3,4,5,6,7,8,9)$ & $K_{10}$
\end{tabular}
\caption{Symmetric isoperimetric sequences of length 10}
\end{center}
\end{table}

For $n=11$ there are $28$ symmetric $\delta$-sequences, out of which only 5
are isoperimetric.

\begin{table}[htb]
\begin{center}
\begin{tabular}{l|l}
$\delta$-sequence & graph \\\hline
$(0,1,2,2,2,3,4,4,4,5,6)$ & $K_{11} - 2C_{11}^*$ \\
$(0,1,2,2,3,3,3,4,4,5,6)$ & previously unknown$^*$\\
$(0,1,2,3,3,4,5,5,6,7,8)$ & previously unknown$^*$\\
$(0,1,2,3,4,4,4,5,6,7,8)$ & $K_{11} - C_{11}^*$ \\
$(0,1,2,3,4,5,6,7,8,9,10)$ & $K_{11}$
\end{tabular}
\caption{Symmetric isoperimetric sequences of length 11}
\end{center}
\end{table}

In \cite{BE} we constructed dense regular graphs $G$ admitting optimal 
orders for $G\times G$ by removing some perfect matchings from $K_p$ or 
$K_{p,p}$. This way $|V_G|$ must be even. To avoid this restriction we
attempted to remove 2-factors from $K_p$. It turns out that for the resulting
graph to admit some optimal order on $G\times G$ it is important which
2-factor to remove. However, in the simplest setting, how about removing a
Hamiltonian cycle from $K_p$? Particularly interesting is the case of odd $p$,
which is not covered by \cite{BE}. This way we came to the graph
$K_p-C_p$ with
  $$\delta_{K_p-C_p} =
    (\{0,1,2,\dots,(p-3)/2\},\;\;\{(p-3)/3\},\;\;\{(p-3)/2,\dots,p-3\}).$$ 
This type of $\delta$-sequence has a ``plateau'' in the middle and is not 
covered by Theorem \ref{main}. However,
\begin{tm} {\rm (DeVries \cite{D})}
Lexicographic order is optimal for $(K_p-C_p)^2$ for odd $p\geq 9$ or $p=5$.
\end{tm}

It is interesting to mention that $(K_7-C_7)^2$ does not admit any optimal
order.

\end{document}